\documentclass[a4paper]{amsart}
\usepackage{amsmath}
\usepackage{amssymb}
\usepackage{graphicx}
\usepackage{color}
\usepackage{tikz} 
\newcommand{\sym}{\mathcal{S}}

\newcommand{\A}{\mathcal{A}}      
\newcommand{\NN}{{\mathbb N}}
\newcommand{\asc}{\mbox{asc}}
\newcommand{\Asc}[1]{\A_{#1}}

\newcommand{\subi}{\textsf{Rem1}}
\newcommand{\subii}{\textsf{Rem2}}
\newcommand{\subiii}{\textsf{Rem3}}

\newcommand{\addi}{\textsf{Add1}}
\newcommand{\addii}{\textsf{Add2}}
\newcommand{\addiii}{\textsf{Add3}}

\newcommand{\cremove}{\psi '}
\newcommand{\cmyadd}{\varphi '}

\newcommand{\tpt}{$(\mathbf{2+2})$}

\newtheorem{proposition}{Proposition}

\newtheorem{theorem}[proposition]{Theorem}

\newtheorem{lemma}[proposition]{Lemma}
\theoremstyle{definition}

\newtheorem{example}[proposition]{Example}

\newcommand{\cclass}[1]{\mathcal{I}_{#1}}

\newcommand{\newl}{\mathsf{M}}
\newcommand{\newsl}{\mathsf{m}}
\newcommand{\red}{\mbox{\sf{redarc}}}
\newcommand{\maxch}{\mbox{\sf{maxarc}}}
\newcommand{\chords}{\mbox{\sf{arcs}}}
\newcommand{\clabel}{\mbox{\sf{label}}}

\newcommand{\cfill}{black!40}

\newcommand{\PATTERN}{
    \draw[step=1, xshift=14pt, yshift=14pt, \cfill] (0,0) grid (3,3);  
    \draw[step=1, xshift=14pt, yshift=14pt, thick] (0,1) -- (3,1);  
    \draw[step=1, xshift=14pt, yshift=14pt, thick] (1,0) -- (1,3);  
    \foreach \x/\y in {1/2,2/3,3/1} \node[disc, fill=black] at (\x,\y) {};  
}
\newcommand{\pattern}{\!\raisebox{-0.5em}{
  \begin{tikzpicture}[line width=0.7pt, scale=0.15]
    \tikzstyle{disc} = [circle,thin,draw=black, minimum size=1.7pt, inner sep=0pt ]
    \PATTERN
  \end{tikzpicture}}
}

\title{A direct encoding of Stoimenow’s matchings as ascent sequences}
\thanks{The authors were supported by grant no. 090038011 
  from the Icelandic Research Fund.}
\keywords{Combinatorial problem, encoding, matching, ascent sequences}

\author[A. Claesson]{Anders Claesson}
\address{A. Claesson and S. Kitaev: The Mathematics Institute,
Reykjavik University, 103 Reykjavik, Iceland.}
\author[M. Dukes]{Mark Dukes}
\address{M. Dukes: Science Institute, University of Iceland, 107
Reykjavik, Iceland.}

\author[S. Kitaev]{Sergey Kitaev}

\begin{document}
\begin{abstract}
In connection with Vassiliev's knot invariants, Stoimenow (1998) introduced
certain matchings, also called regular linearized chord diagrams.
Bousquet-M\'elou et al. (2008) gave a bijection from those
matchings to unlabeled $(\mathbf{2+2})$-free posets; they also showed how to
encode the posets as so called ascent sequences. In this paper we
present a direct encoding of Stoimenow's matchings as ascent
sequences. In doing so we give the rules for recursively constructing
and deconstructing such matchings.
\end{abstract}
\maketitle

\thispagestyle{empty}

\section{Introduction}

To give upper bounds on the dimension of the space of Vassiliev's knot
invariants of a given degree, Stoimenow~\cite{stoim} introduced what
he calls regular linearized chord diagrams. We call them {\it
  Stoimenow matchings}. As an example, these are the 5 Stoimenow
matchings on the set $\{1,2,3,4,5,6\}$:
$$
\scalebox{0.75}{\includegraphics{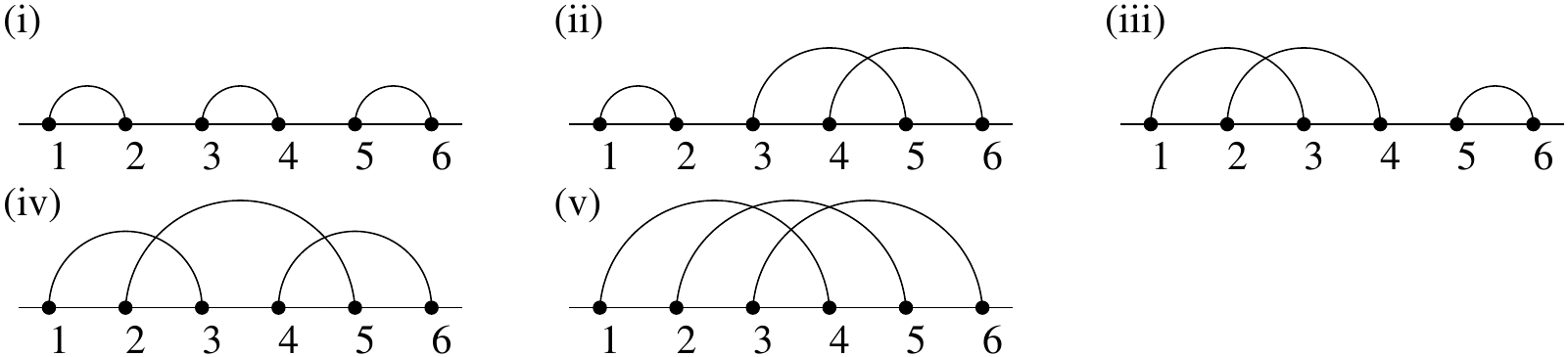}}
$$ 
In general, a \emph{matching} of the integers $\{1,2,\ldots , 2n\}$ is a
partition of that set into blocks of size 2, often called \emph{arcs}.
We say that a matching is \emph{Stoimenow} if there
are no occurrences of Type 1 or Type 2 arcs:
$$
\scalebox{0.75}{\includegraphics{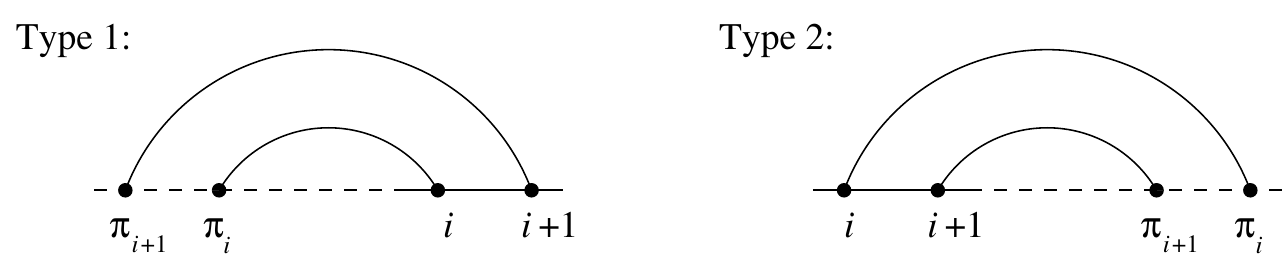}}
$$

In this paper we present a bijection between Stoimenow matchings on
$\{1,2,\ldots,2n\}$ and a collection of sequences of non-negative
integers that we call {\it{ascent sequences}}.

Given a sequence of integers $x=(x_1,\ldots , x_n)$, we say that the
sequence $x$ has an {\it{ascent}} at position $i$ if $x_i<x_{i+1}$.
The number of ascents of $x$ is denoted by $\asc(x)$.  Let
$\mathcal{A}_{n}$ be the collection of {\it{ascent sequences of length
    $n$}}:
$$\Asc{n} = \big\{\,(x_1,\ldots , x_n): x_1=0\text{ and } 
        0 \leq x_i \leq 1+\asc(x_1,\ldots , x_{i-1})
        \mbox{ for }1<i\leq n \,\big\}.
$$
These sequences were introduced in a recent paper by
Bousquet-M\'elou et al.~\cite{bcdk}. For example,
$\Asc{3}=\{ (0,0,0), \, (0,0,1), \, (0,1,0), \, (0,1,1), (0,1,2)
\}.$

Bousquet-M\'elou et al. gave bijections between four classes of
combinatorial objects, thus proving that they are equinumerous:
Stoimenow matchings; unlabeled \tpt-free posets; permutations avoiding
a specific pattern; and ascent sequences.  The following diagram, in
which solid arrows represents bijections given by Bousquet-M\'elou et
al., sums up the situation.
\begin{center}
  \begin{tikzpicture}
    \node (matching) at (0, 0) {Stoimenow matchings};
    \node (poset)    at (0, 2) {unlabeled \tpt-free posets};
    \node (ascseq)   at (6, 2) {ascent sequences};
    \node (perm)     at (6, 0) {\pattern-avoiding permutations};
    \path[->] (matching) edge node[left]  {$\Omega$}  (poset);
    \path[->] (poset)    edge node[above] {$\Psi$}    (ascseq);
    \path[->] (perm)     edge node[right] {$\Lambda$} (ascseq);
    \path[->, dashed] (matching) edge node [below] {$\Psi'$} (ascseq);
  \end{tikzpicture}
\end{center}
In particular, $\Psi\circ\Omega$ is a bijection between Stoimenow
matchings and ascent sequences. The dashed arrow is the contribution
of this paper. That is, we give a direct description of
$\Psi'=\Psi\circ\Omega$. Ascent sequences have an obvious recursive
structure. We unearth the corresponding recursive structure of
Stoimenow matchings. It amounts to two functions, $\varphi'$ and
$\psi'$, which act on matchings in an identical manner to the
functions $\varphi$ and $\psi$ of \cite[\S 3]{bcdk} acting on posets.

\section{Stoimenow matchings and edge removal}

Let $\cclass{n}$ be the collection of Stoimenow matchings with $n$
arcs. Let $\sym_{m}$ be the collection of all permutations of the set
$\{1,\ldots , m\}$. Any Stoimenow matchings may be written uniquely as
a fixed point free involution $\pi \in \sym_{2n}$ so that the number
paired with $i$ is $\pi_i$. We shall abuse notation ever so slightly
by considering $\pi$ to be dually a matching in $\cclass{n}$ and an
involution in $\sym_{2n}$.

Given $\pi \in \cclass{n}$ let $\chords(\pi)$ be the collection of all
$n$ arcs $[i,\pi_i]$ of $\pi$.  Let us introduce the following
labelling scheme $\clabel: \chords(\pi) \to \NN$ of the arcs; for
every arc in $\pi$, call the left endpoint the {\em{opener}} and the
right endpoint the {\em{closer}}.

\begin{minipage}{\textwidth}
\begin{minipage}{8cm}
Label an arc with the number of runs of closers that precede it. For
example, consider the matching $\{[1,3],[2,4],[5,6]\}$, or
equivalently the involution $\pi=341265$. The labels of the arcs are
shown in the diagram to the right.
\end{minipage}
\begin{minipage}{5cm}
\begin{center}
\scalebox{0.8}{\includegraphics{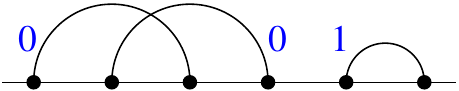}}
\end{center}
\end{minipage}
\end{minipage}

To every Stoimenow matching we shall single out two (very important)
arcs.  Given $\pi \in \cclass{n}$ call $\maxch(\pi)=[\pi_{2n},2n]$ the
{\em{maximal arc}} of $\pi$ and call the arc
$\red(\pi)=[\pi_{1+\pi_{2n}}, 1+\pi_{2n}]$ the {\em{reduction arc}} of
$\pi$.

To every Stoimenow matching we shall associate two statistics:
$$\newl(\pi)= \clabel(\maxch(\pi))\quad\text{and}\quad\newsl(\pi)=\clabel(\red(\pi)).
$$ 
For the matching $\{[1,3],[2,4],[5,6]\}$ we have
$\red(\pi)=[5,6]=\maxch(\pi)$ so that $\newl(341265)=1$ and
$\newsl(341265)=1$.  In the diagrams that follow, vertices that are
openers are marked with a $\bullet$ and closers are marked with a
$\Box$.

\begin{example}\label{examp:labels}
\ \\
\vspace*{-3em}
\begin{enumerate}
\item[(i)] 
\begin{minipage}[t]{14cm}
\begin{minipage}[t]{6cm}
Consider $\pi \, = \,3\,4\,1\,2\,7\,9\,5\,10\,6\,8\,\in\,\cclass{5}$.\\
  We have $\red(\pi)=[6,9]$\\ and $\maxch(\pi)=[8,10]$.\\
  This gives $\newsl(\pi)=1$ and $\newl(\pi)=2$.
\end{minipage}
\begin{minipage}[c]{5cm}
\ \\[1em]
\begin{center}
\scalebox{0.8}{\includegraphics{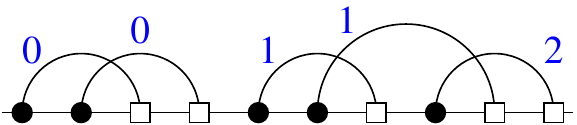}}
\end{center}
\end{minipage}
\end{minipage}
\vspace*{-2em}
\item[(ii)]
\begin{minipage}[t]{14cm}
\begin{minipage}[t]{6cm}
  Consider $\pi\,=\,4\,5\,7\,1\,2\,8\,3\,6 \,10\,9 \, \in \cclass{5}$.
  The labels of the arcs are shown in the diagram.
  We have $\red(\pi)=[9,10]=\maxch(\pi)$ and so $\newl(\pi)=2= \newsl(\pi)=2$.
\end{minipage}
\begin{minipage}[c]{5cm}
\ \\[1em]
\begin{center}
  \scalebox{0.8}{\includegraphics{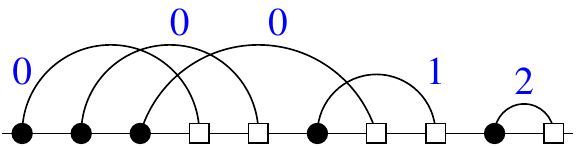}}
\end{center}
\end{minipage}
\end{minipage}
\end{enumerate}
\end{example}

We begin with the removal operations for Stoimenow matchings.  Let
$\pi \in \cclass{n}$ where $n\geq 2$ and let $i=\newsl(\pi)$ be the
label of the reduction arc $\red(\pi)$.  In what follows we will
remove the reduction arc in a very careful way so that we obtain
$\sigma \in \cclass{n-1}$.

Let $L_i(\pi)=\{x \in \chords(\pi): \clabel(x)=i\}$ be the set of arcs
that have label $i$.

\begin{enumerate}
\item[(\subi)] If $|L_i(\pi)|>1$ then simply remove the reduction arc
  $\red(\pi)$.
\item[(\subii)] If $|L_i(\pi)|=1$ and $i=\newl(\pi)$, then
  $\maxch(\pi)=\red(\pi)=[2n-1,2n]$ and we remove this arc from $\pi$.
\item[(\subiii)] If $|L_i(\pi)|=1$ and $i<\newl(\pi)$ then do as
  follows (these steps are illustrated in Figure \ref{ruleremt};
  \begin{enumerate}
  \item[(a)] Let $A$ be the collection of all closers between $x$ and
    the next opener to its right.  Move all points in $A$ to between
    $z$ and $u$ while respecting their order relative to one-another.
  \item[(b)] For all $j$ with $0\leq j<i$, partition the collection of
    openers with label $j$ into three segments $X_j,Y_j$ and $Z_j$
    where $Y_j$ is the collection of openers that have closers in $A$.
    Swap each of the sets $Y_j$ and $Z_j$ while preserving their
    respective internal order.
  \item[(c)] Remove the reduction arc.
  \end{enumerate}
\end{enumerate}

\begin{figure}
\centerline{\scalebox{0.8}{\includegraphics{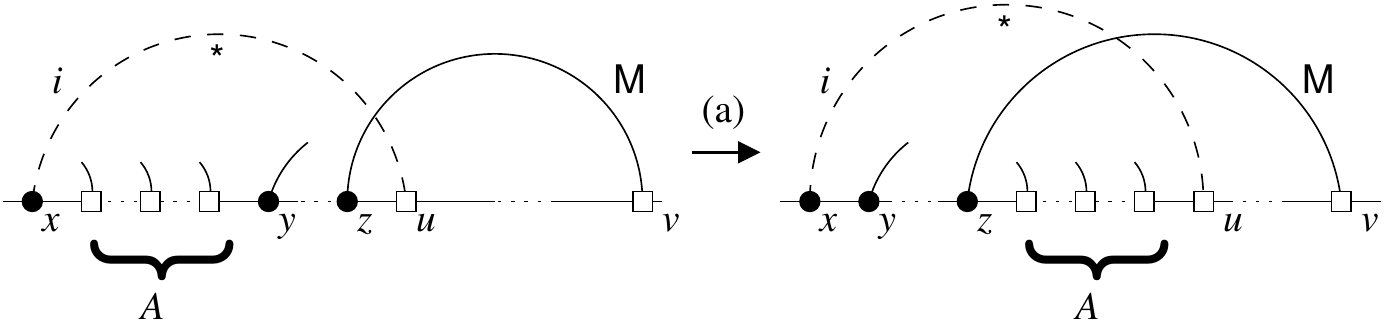}}}
\ \\ 
\centerline{\scalebox{0.8}{\includegraphics{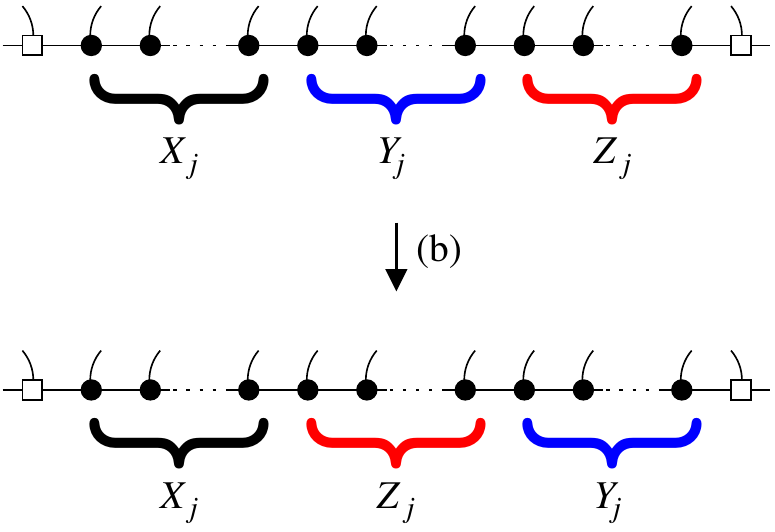}}}
\caption{The removal rule {\subiii}.}
\label{ruleremt}
\end{figure}

\begin{example}
Three examples corresponding to the above removal operations.
\begin{enumerate}
\item[(i)]
In Example~\ref{examp:labels}(i) we had $i=\newsl(\pi)=1$, $\newl(\pi)=2$ and $|L_2(\pi)|=2>1$. Thus rule (\subi) applies and we have $\sigma$:\\[1em]
\centerline{\scalebox{0.8}{\includegraphics{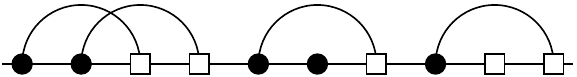}}}\\[1em]
\item[(ii)] In Example~\ref{examp:labels}(ii) we had $i=\newsl(\pi)=2=\newl(\pi)$ and $|L_2(\pi)|=1$. Thus rule (\subii) applies and we have $\sigma$:\\[1em]
\centerline{\scalebox{0.8}{\includegraphics{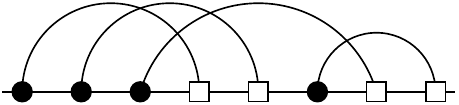}}}

\item[(iii)] 
See Figure \ref{large_example}.
\end{enumerate}
\end{example}

\begin{figure}
\centerline{\scalebox{0.75}{\includegraphics{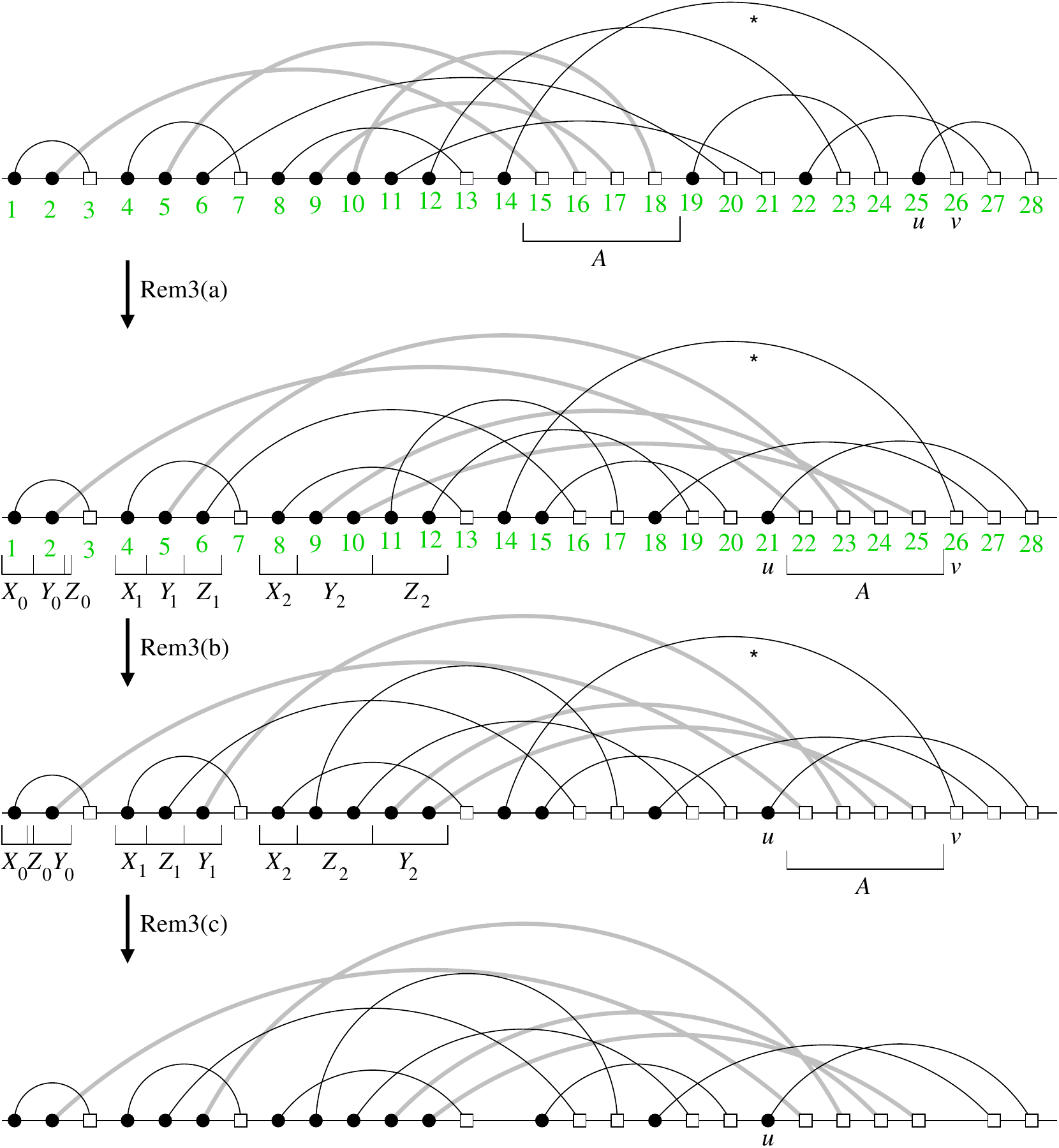}}}
\caption{Illustration of the 3 steps for {\subiii}.}
\label{large_example}
\end{figure}

\begin{example}
See Figure \ref{full_rem_decon} for an example of transforming
a Stoimenow matching into an ascent sequence.
\end{example}

\begin{figure}
\begin{center}
{\scalebox{0.75}{\includegraphics{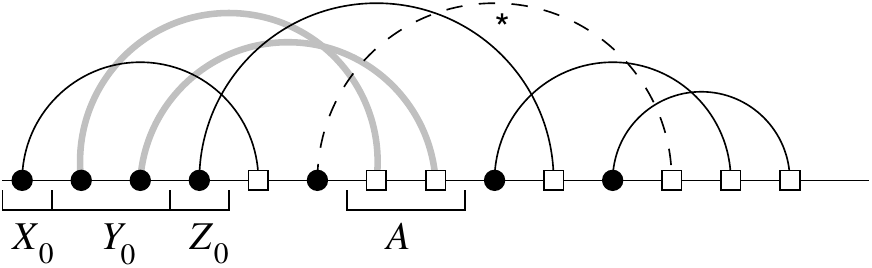}}}\\[1em]
$\downarrow $ (\subiii)  $x_7=1$\\[1em]
{\scalebox{0.75}{\includegraphics{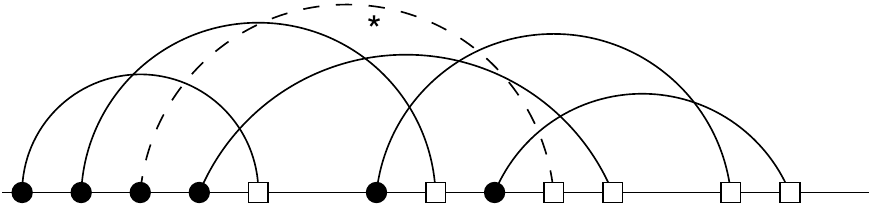}}}\\[1em]
$\downarrow $ (\subi) $x_6=0$\\[1em]
{\scalebox{0.75}{\includegraphics{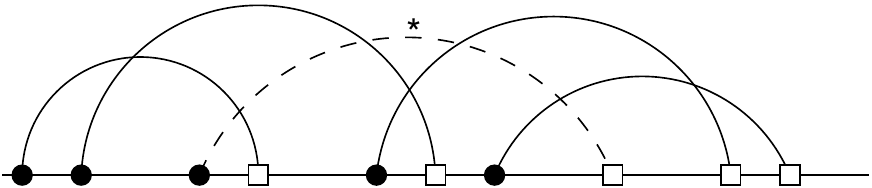}}}\\[1em]
$\downarrow $ (\subi) $x_5=0$\\[1em]
{\scalebox{0.75}{\includegraphics{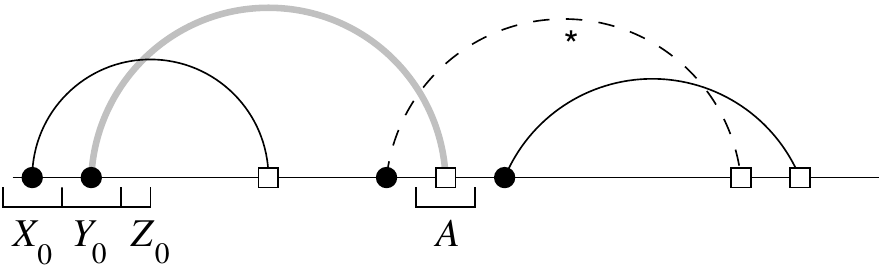}}}\\[1em]
$\downarrow $ (\subiii) $x_4=1$\\[1em]
{\scalebox{0.75}{\includegraphics{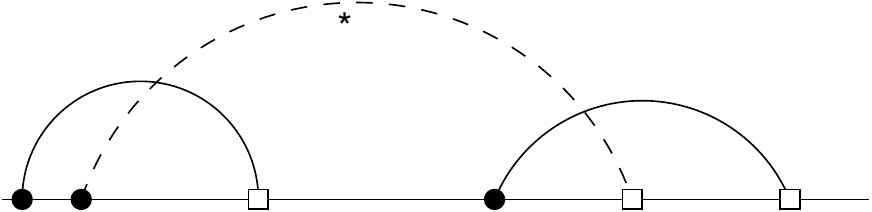}}}\\[1em]
$\downarrow $ (\subi) $x_3=0$\\[1em]
{\scalebox{0.75}{\includegraphics{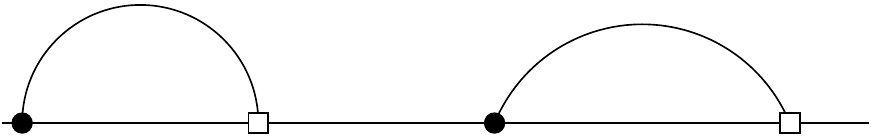}}}\\[1em]
$\downarrow $ (\subii) $x_2=1$ \\[1em]
{\scalebox{0.75}{\includegraphics{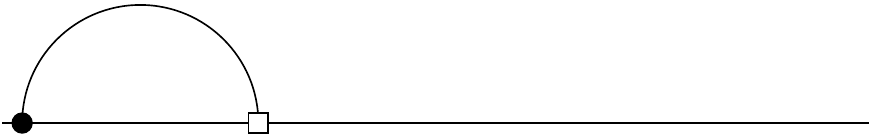}}}
\end{center}
\caption{Using the removal operations to go from the Stoimenow
  matching $\pi \,=\, (5,7,8,10,1,12,2,3,13,4,14,6,9,11) \in
  \cclass{7}$ to the ascent sequence $x=(0,1,0,1,0,0,1)$.}
\label{full_rem_decon}
\end{figure}

We will now prove that the three types of removal operation give 
some $\sigma \in \cclass{n-1}$. If $\newsl(\pi)=i$ and 
the removal operation, when applied to $\pi$ gives $\sigma$,
then define $\psi'(\pi)=(\sigma,i)$.

\begin{lemma}\label{arc:remove}
If $n\geq 2$, $\pi \in \cclass{n}$ and $\psi'(\pi)=(\sigma,i)$ then
$\sigma \in \cclass{n-1}$ and $0\leq i \leq 1+\newl(\pi)$. Also,
\begin{eqnarray*}
\newl(\sigma) &=&
\left\{
	\begin{array}{ll}
	\newl(\pi) & \mbox{ if } i \leq \newsl(\sigma),\\
	\newl(\pi)-1 & \mbox{ if } i > \newsl(\sigma).
	\end{array}
\right.
\end{eqnarray*}
\end{lemma}

\begin{proof}
In this proof we show that each of the three removal operations, when
applied to a Stoimenow matching, produce another Stoimenow matching.
The removal of an arc from a Stoimenow matching produces a matching,
but it is necessary to show the matching is Stoimenow, i.e. does not
contain type 1 or type 2 arcs.

In both {\subi} and {\subii} we are simply deleting the reduction arc.
Thus the only neighbouring points to check the Stoimenow property (no
type 1 or type 2 nestings) are the pairs of points adjacent to the
left and right endpoints of $\red(\pi)$.  However for {\subiii} the
situation is slightly more complicated.

In the diagrams, lozenge vertices $\lozenge$ correspond to points
which could be openers or closers and the reduction arc is indicated
by $\star$.

\begin{minipage}{\textwidth}
\begin{minipage}{7cm}
(\subi) In this case $|L_i(\pi)|>1$.  We must check that the removal
of the reduction arc $\red(\pi)$ does not introduce a type 1 or 2 arc
in $\sigma$.  If $\newsl(\pi)=\newl(\pi)$ then we have the 
situation as indicated to the right.
\end{minipage}
$\hspace*{1em}$
\begin{minipage}{7cm}
{\scalebox{0.8}{\includegraphics{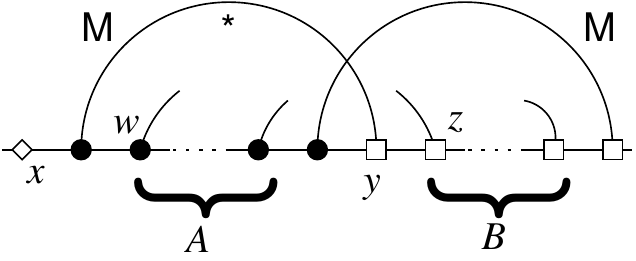}}}
\end{minipage}
\end{minipage}

If the
set of points $A$ is empty then the set $B$ must be empty, for
otherwise the arcs with endpoints $y$ and $z$ are type 1.  By the same
argument, if $B$ is empty then so is $A$.  This gives the following
situation if $A=B=\emptyset$:

\begin{minipage}{\textwidth}
\begin{minipage}{7cm}
{\scalebox{0.8}{\includegraphics{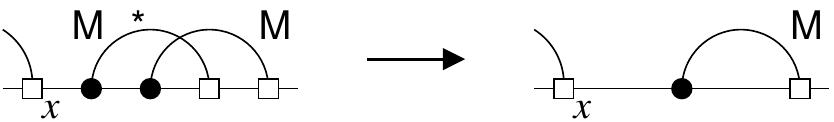}}}
\end{minipage}
$\hspace*{1em}$
\begin{minipage}{5.1cm}
The
point $x$ must be a closer since there are no available points to its
right.  Removing the reduction arc preserves the Stoimenow property.
\end{minipage}
\end{minipage}

Alternatively, $A$ is not empty iff $B$ is not empty. In fact all arcs
with opener in $A$ have a closer in $B$.  Similarly, all closers in
$B$ have openers in $A$ (for otherwise a type 1 arc would appear).
Also, $x$ must be a closer, for otherwise a type 2 arc arises with $x$
and $w$ as endpoints.  We have the following situation:\\[1em]
\centerline{\scalebox{0.8}{\includegraphics{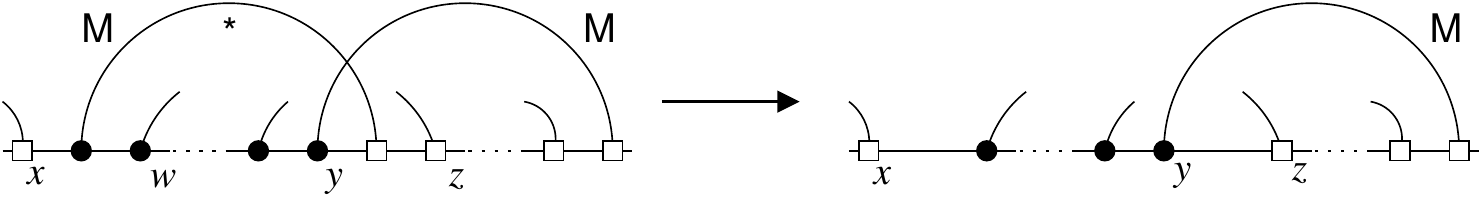}}}\\[1em] It
is straightforward to see that the removal of the opener of the
reduction arc preserves the Stoimenow property.

\begin{minipage}{\textwidth}
\begin{minipage}{7cm}
If $\newsl(\pi)<\newl(\pi)$ then there are at least 2 arcs with label
$\newsl(\pi)$.  Consequently, at least one of $x$ and $y$ in the
following diagram must be an opener. Also, note that there must be a
closer between the openers of $\red(\pi)$ and $\maxch(\pi)$.
\end{minipage}
$\hspace*{1em}$
\begin{minipage}{7cm}
\scalebox{0.8}{\includegraphics{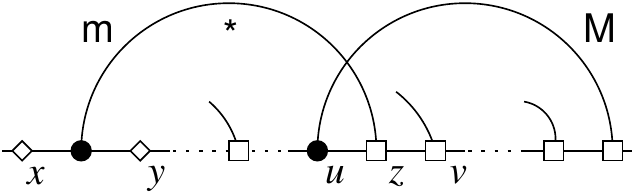}}
\end{minipage}
\end{minipage}

First note that the Stoimenow property is preserved at the newly
adjacent points $u$ and $v$ once $z$ is removed.  Next, if $x$ is a
closer then $y$ must be an opener. Thus removing the opener of
$\red(\pi)$ preserves the Stoimenow property at (the newly adjacent
points) $x$ and $y$.

If $x$ is an opener, then $y$ can be an opener or a closer.  In the
case that $y$ is a closer, then the Stoimenow property is preserved.
If $x$ and $y$ are both openers then the endpoint of $y$ must be to
the left of $z$ since it is a Stoimenow matching.  Similarly, the
endpoint of $x$ must be to the left of $u$. Thus the Stoimenow
property is preserved at the newly adjacent points $x$ and $y$.  In
the arguments above, the label of the maximal arc remains unchanged,
hence $\newl(\pi)=\newl(\sigma)$.

(\subii) In this case $|L_i(\pi)|=1$ and $i=\newl(\pi)$.  There is a
unique arc $[2n-1,2n]$ in $\pi$ that has maximal label $\newl(\pi)$.
Removing this arc of course yields $\sigma \in \cclass{n-1}$. Since
this arc does not cross any other arcs in the diagram, its removal
cannot induce a type 1 or type 2 arc.  It was the only arc with label
$\newl(\pi)$ so we have $\newl(\sigma)=\newl(\pi)-1$.

(\subiii) In this case $|L_i(\pi)|=1$ and $i<\newl(\pi)$.  We must
check that the matching obtained after performing operations (a), (b)
and (c) is Stoimenow.  It is not necessarily true that the matching is
Stoimenow after performing (a).  The combination of (a), (b) and (c)
is needed to ensure the Stoimenow property.  See Figure
\ref{general_rem3} for an illustration of {\subiii}.

Note that $0\leq j<i$.  Let $A$ be the run of closers between the
opener of the reduction arc and the next opener to its right.  Let $B$
be the segment whose leftmost point is the opener to the right of $A$
and whose rightmost point is the opener of the maximal arc.  Let $C$
be the run of closers that is to the right of the closer of the
reduction arc, and to the left of the rightmost closer.

There are only certain places in $\sigma$ where the Stoimenow property
may have been broken.  The segments of openers $W_j=(X_j,Y_j,Z_j)$ in
$\pi$ are separated by closers.  Thus no two arcs from two different
$W_k$'s (where $k<i$) can form a type 2 pair in $\sigma$.  Hence we
may restrict our attention to one segment of openers $W_j$ and the
action of steps (a), (b) and (c) on this segment and its interaction
with $A$, $B$ and $C$.

After {\subiii} has been applied, the internal order of each of the
$X_j$, $Y_j$ and $Z_j$ segments is unaltered.  Thus the Stoimenow
property cannot be broken within each of these segments.  However, the
order in which the segments $Y_j$ and $Z_j$ appear in $\sigma$ has
been transposed.  Similarly, the segments $A$, $B$ and $C$ retain
their internal order so that the Stoimenow property is not violated
within each.

By this reasoning there are only six cases to consider where the Stoimenow property may be broken. 
These are indicated by roman numerals in Figure \ref{general_rem3}.

\begin{figure} 
\centerline{\scalebox{0.8}{\includegraphics{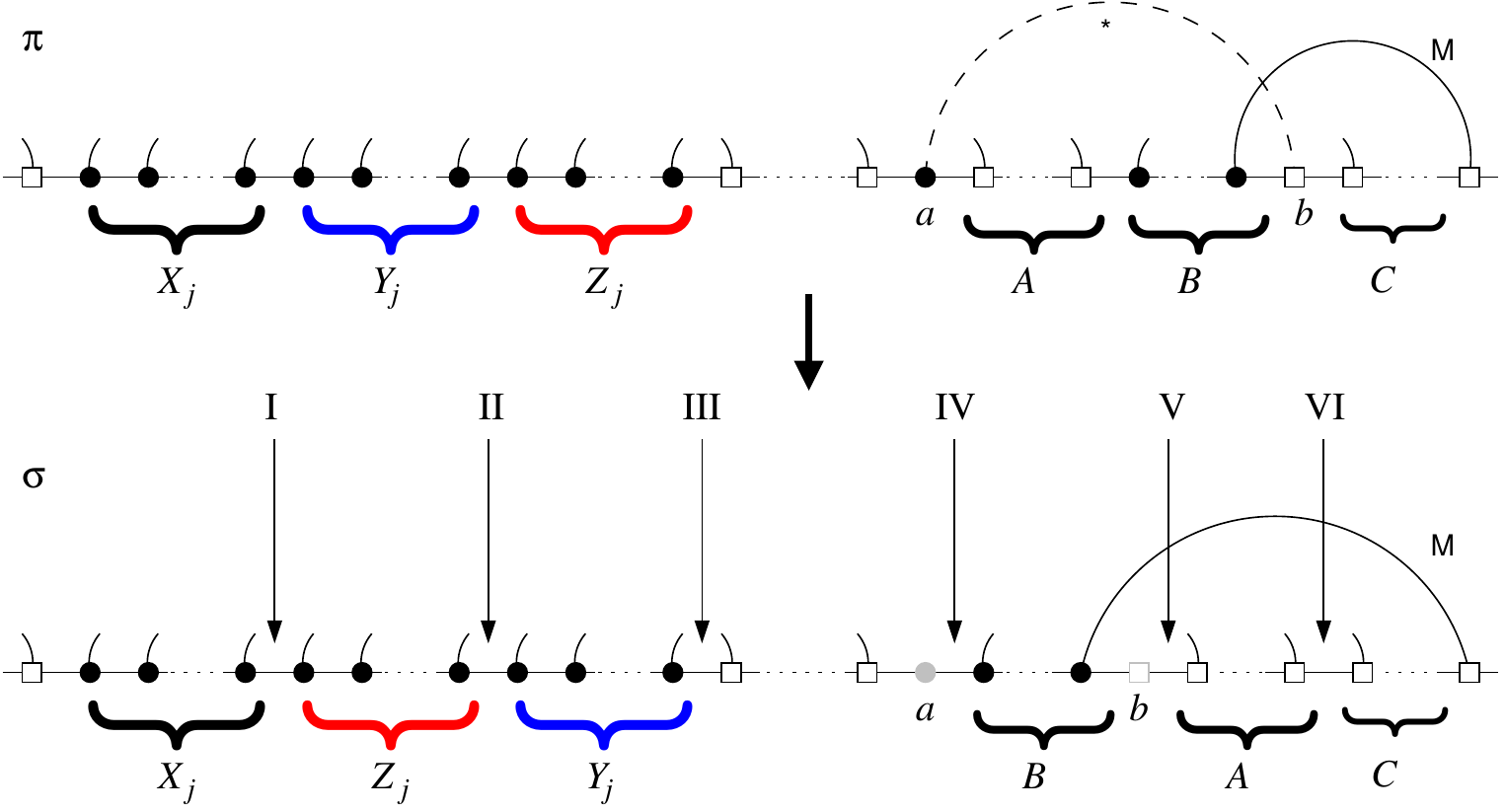}}}
\caption{}
\label{general_rem3}
\end{figure}

The adjacent points in cases III, IV and V are such that one point is
an opener and the other is a closer, thereby preserving the Stoimenow
property at these positions.

\begin{itemize}
\item[(I)] If $X_j$ is empty then there is no opener immediately to
  the left of $Z_j$ in $\sigma$ with which to form a type 2 arc.
  Otherwise $X_j$ is not empty and in $\sigma$ the closers
  corresponding to $X_j$ are located to the left of $a$, whereas
  closers corresponding to $Z_j$ are in $B$ which is to the right of
  $a$.  Hence the Stoimenow property is preserved.

\item[(II)] If both $Z_j$ and $X_j$ are empty then there is no opener
  immediately to the left of $Y_j$.  If $Z_j$ is empty and $X_j$ is
  not empty then the closers of $X_j$ are to the left of $a$ and the
  closers of $Y_j$ are to the right of $a$.  If $Z_j$ is not empty
  then the closers corresponding to $Z_j$ are in $B$.  The closers
  corresponding to $Y_j$ are in $A$.  Since $A$ is to the right of $B$
  in $\sigma$, the new neighbors do not form a prohibited type 2 arc.

\item[(VI)] If $C$ is empty then we have two adjacent closers at the
  end of $\sigma$.  The opener corresponding to the rightmost opener
  of $A$ is to the left of $b$ so the Stoimenow property is preserved.
  Otherwise $C$ is non-empty and the openers corresponding to $A$ are
  $Y_j$, whereas the openers corresponding to $C$ are in $B$. Since
  $Y_j$ is to the left of $B$, the new neighbors do not form the
  prohibited type I arc.
\end{itemize}

The arc that was removed was the only arc in $\pi$ with label $i$, so
$\newl(\sigma)=\newl(\pi)-1$.
\end{proof}

\section{Adding an edge to a Stoimenow matching}
We now define the addition operation for Stoimenow matchings.
Given $\sigma \in \cclass{n-1}$ and $0\leq i \leq 1+\newl(\sigma)$,
let $\varphi ' (\sigma , i)$ be the Stoimenow matching $\pi$ obtained from $\sigma$
according to the following addition rules

\begin{enumerate}
\item[(\addi)] If $i\leq \newsl(\sigma)$ then partition the segment 
of openers with label $i$ into two (possibly empty) segments:
let $A$ be the contiguous segment of openers which do not intersect the maximal 
arc and let $B$ be the contiguous segment of openers that do intersect the maximal arc. 
Note the $A$ is always to the left of $B$.
Insert an arc by introducing a new point between $A$ and $B$, and another new point immediately to the right of $\pi_{n-1}$.
(See Figure \ref{ruleaddone}.)

\item[(\addii)] If $i=1+\newl(\sigma)$ then introduce the arc $[2n-1,2n]$ to $\sigma$.

\item[(\addiii)] If $\newsl(\sigma) < i \leq \newl(\sigma)$ then do as follows (each of these steps is illustrated in Figure \ref{ruleaddt})
\begin{enumerate}
	\item[(a)] Locate the first opener of $\sigma$ with label $i$ and call it $d$. 
		Insert an imaginary vertical line $L$ in the diagram just before $d$.
		Let $A$ be the contiguous segment of closers immediately right of the opener of the maximal arc, $c$, 
		whose openers lie to the right of $L$. 
		Let $C$ be the segment of points after $A$ and before the rightmost point of $\sigma$.
		Insert two new points: one where the line $L$ crosses the diagram, $a$, and another in-between $A$ and $C$, $b$. 
		Join these points by an arc.
	\item[(b)] For all $0\leq j <i$, partition the segments of openers with label $j$ into three segments $X_j,Y_j$ and $Z_j$.
			The arcs from $X_j$ have closers that lie to the left of $L$.
			The arcs from $Y_j$ have closers that are in $A$. $Z_j$ is what remains.
			Swap the segments $Z_j$ and $Y_j$ for each $j$ while preserving the internal order of
			the openers.
	\item[(c)] Finally, move the segment of closers $A$ in-between the points $a$ and $d$.
\end{enumerate}
\end{enumerate}

\begin{figure}
\centerline{\scalebox{0.8}{\includegraphics{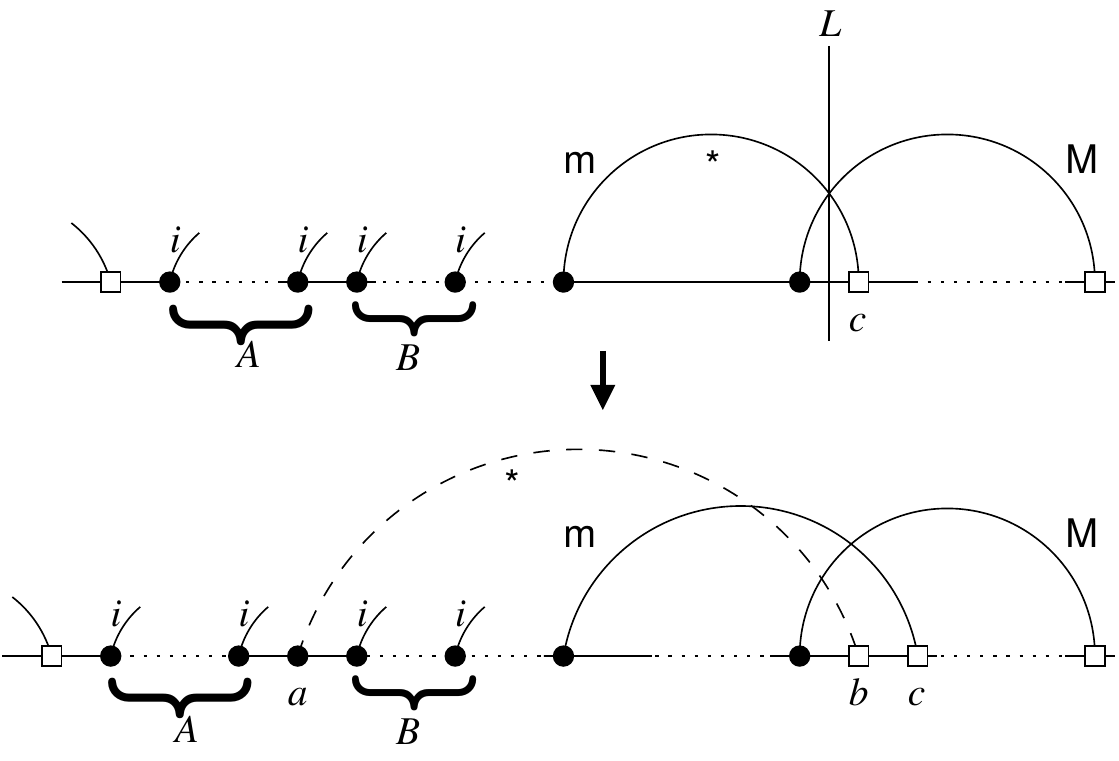}}}
\caption{The addition rule {\addi}.}
\label{ruleaddone}
\end{figure}

\begin{figure}
\centerline{\scalebox{0.75}{\includegraphics{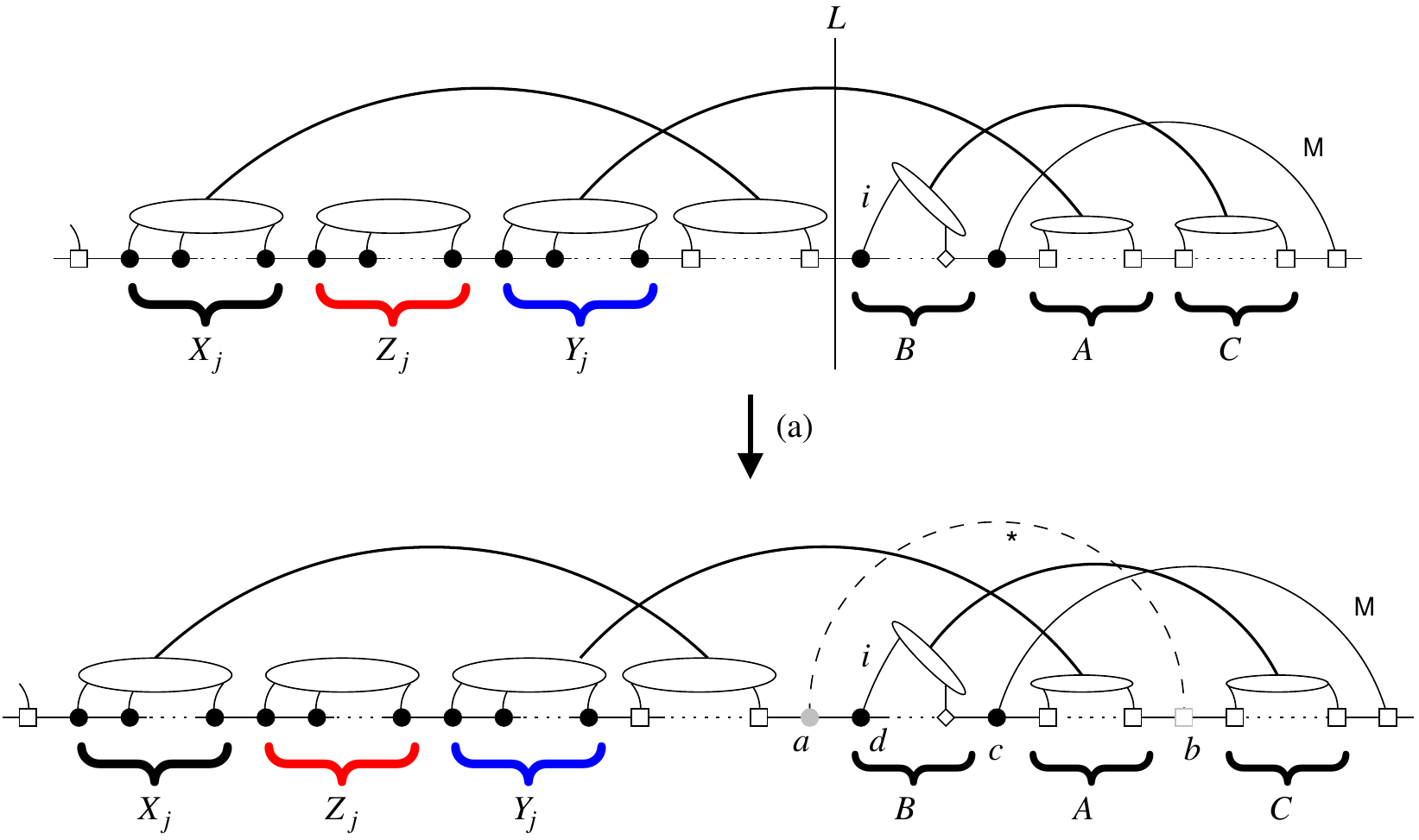}}}
\ \\ 
\centerline{\scalebox{0.75}{\includegraphics{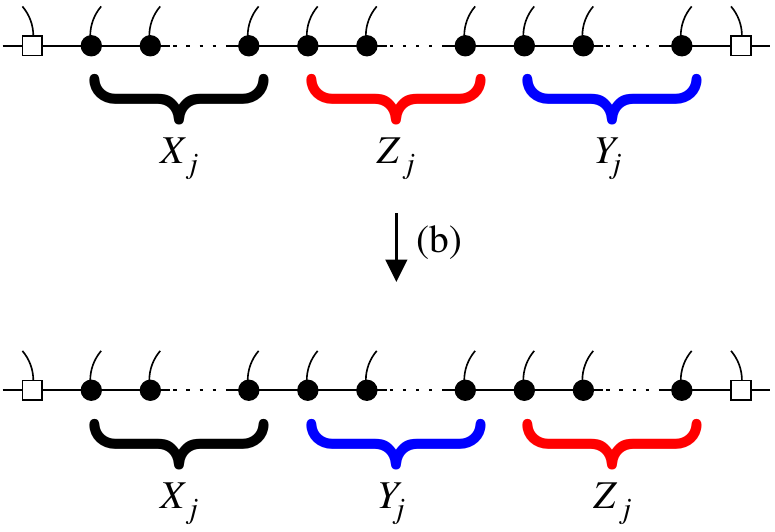}}}
\ \\ 
\centerline{\scalebox{0.75}{\includegraphics{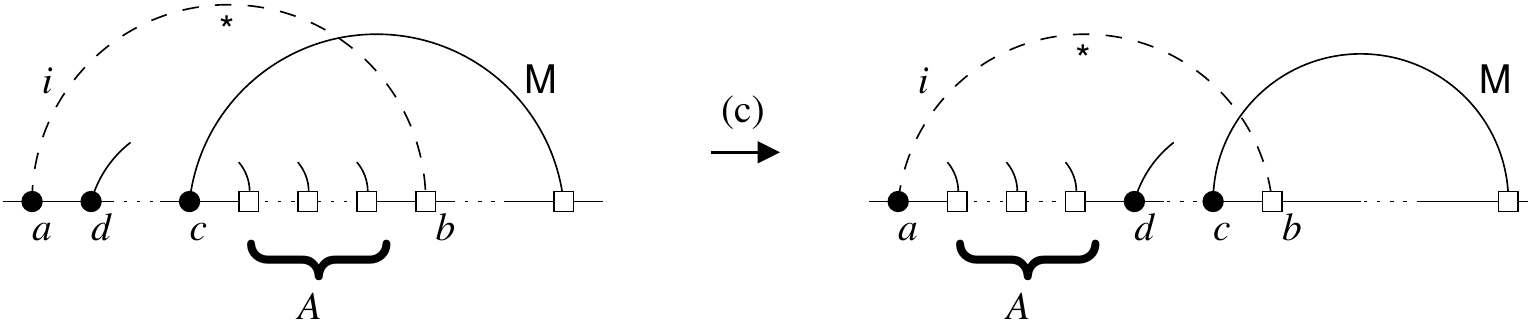}}}
\caption{The addition rule {\addiii}.}
\label{ruleaddt}
\end{figure}

\begin{lemma}
If $n\geq 2$, $\sigma \in \cclass{n-1}$, $0\leq i \leq 1+\newl(\sigma)$ and $\pi=\varphi '(\sigma, i)$ then
$\pi \in \cclass{n}$. Also, 
 $$
  \newl(\pi) =
  \begin{cases}
    \newl(\sigma)   & \text{if } i \leq \newsl(\sigma),\\
    \newl(\sigma)-1 & \text{if } i > \newsl(\sigma).
  \end{cases}
  $$
\end{lemma}

\begin{proof}
The proof of this requires examining the three addition operations separately.

For the first case $i\leq \newsl(\sigma)$ and $\addi$ is used.  This
is illustrated in Figure \ref{ruleaddone}.  It introduces a new arc
(the end points of this arc are $a$ and $b$ in the figure) which has
label $i$ and serves as the new reduction arc.  Since arcs with
openers in $B$ have closers to the right of $c$, and arcs with openers
in $A$ have closers to the left of $c$, the Stoimenow property is
preserved.  It is easy to see from the diagram that the Stoimenow
property is preserved.  Furthermore, since this new arc is essentially
a copy of arcs with label $i$, $\newl(\pi)=\newl(\sigma)$.

If $i=1+\newl(\sigma)$ then {\addii} is used. 
A new arc is added to the matching on the right hand side. 
This arc does not meet any other arcs so it retains 
the property of being Stoimenow. Also, the label of this arc
will be one more than $\newl(\sigma)$ so that $\newl(\pi)=\newl(\sigma)+1$.

If $\newsl(\sigma)<i\leq \newl(\sigma)$ then {\addiii} is used.  The
details of this part of the proof are the same as the final part of
\cite[Lemma 4]{bcdk}, re-written in the language of matchings as in
Lemma \ref{arc:remove}.
\end{proof}

The machinery has now been set up so that we can see that the recursive structure
of Stoimenow matchings is isomorphic to that of ascent sequences.
We omit the formal proof by induction of the following result which 
gives the compatibility of the removal and addition operations for Stoimenow matchings.

\begin{lemma}
  For any Stoimenow matching $\sigma$ and integer $i$ such that $0\leq
  i\leq 1+\newl(Q)$ we have $\cremove(\cmyadd(\sigma,i)) =
  (\sigma,i)$.  If $\sigma$ has more than one element then we also
  have $\cmyadd(\cremove(\sigma)) = \sigma$.
\end{lemma}

\section{Stoimenow matchings to ascent sequences}
Define the map $\Psi' :\mathcal{I}_n \to \mathcal{A}_n$ as follows.
For $n=1$ we associate the only Stoimenow matching in $\cclass{1}$
with the sequence $(0)$.  Let $n\geq 2$ and suppose that the removal
operation, applied to $\pi \in \cclass{n}$ gives
$\cremove(\pi)=(\sigma,i)$.  Then the sequence associated with $\pi$
is $\Psi '(\pi) := (x_1,\ldots , x_{n-1},i)$ where $(x_1,\ldots ,
x_{n-1})=\Psi '(\sigma)$. 
Combining the previous lemmas, we have the following theorem that is easily proved by induction.
\begin{theorem}
The map $\Psi '$ is a one-to-one correspondence between Stoimenow matchings with $n$
arcs and ascent sequences of length $n$.
\end{theorem}


\begin{thebibliography}{9}

\bibitem{bcdk} M. Bousquet-M\'elou, A. Claesson, M. Dukes and
  S. Kitaev, $\mathbf{(2+2)}$-free posets, ascent sequences
  and pattern avoiding permutations, {\texttt{arXiv:0806.0666}}.

\bibitem{stoim} A. Stoimenow, Enumeration of chord diagrams and an
  upper bound for Vassiliev invariants, {\em{J. Knot Theory
      Ramifications}} {\bf{7}} no. 1 (1998) 93--114.

\end{thebibliography}
\end{document}